\documentclass{amsart}

\usepackage{amsrefs}
\usepackage[all]{xy}
\usepackage{syntonly}
\usepackage{enumerate}
\usepackage{hyperref}
\usepackage{amsfonts}
\usepackage{amssymb}
\usepackage{amsmath}
\usepackage{amsthm}
\usepackage{enumitem}
\usepackage{tikz}
\usepackage{graphics}
\usepackage{graphicx}
\usepackage{color}
\usepackage{comment}

\usepackage{lineno}

\usepackage{multirow}

\newtheorem{theorem}{Theorem}
\newtheorem{corollary}[theorem]{Corollary}

\newtheorem{definition}[theorem]{Definition}
\newtheorem{lemma}[theorem]{Lemma}
\newtheorem{globalClaim}[theorem]{Claim}

\newtheorem{question}[theorem]{Question}

\newtheorem{fact}[theorem]{Fact}

\newtheorem{remark}[theorem]{Remark}

\specialcomment{com}{ \color{red} }{\color{black}} 
\specialcomment{oldcom}{ \color{brown} }{\color{black}} 

\excludecomment{oldcom} 

\begin{document}
\title[Chang's Conjecture and semiproperness of nonreasonable posets]{Chang's Conjecture and semiproperness of nonreasonable posets}

\author{Sean D. Cox}
\email{scox9@vcu.edu}
\address{
Department of Mathematics and Applied Mathematics \\
Virginia Commonwealth University \\
1015 Floyd Avenue \\
Richmond, Virginia 23284, USA 
}



\thanks{The author gratefully acknowledges support from the VCU Presidential Research Quest Fund.}

\begin{abstract}

Let $\mathbb{Q}$ denote the poset which adds a Cohen real then shoots a club through the complement of $\big( [\omega_2]^\omega \big)^V$ with countable conditions.  We prove that the version of Strong Chang's Conjecture from \cite{MR2965421} implies semiproperness of $\mathbb{Q}$, and that semiproperness of $\mathbb{Q}$---in fact semiproperness of any poset which is sufficiently \emph{nonreasonable} in the sense of Foreman-Magidor~\cite{MR1359154}---implies the version of Strong Chang's Conjecture from \cite{MR2723878} and \cite{MR1261218}.  In particular, semiproperness of $\mathbb{Q}$ has large cardinal strength, which answers a question of Friedman-Krueger~\cite{MR2276627}.  One corollary of our work is that the version of Strong Chang's Conjecture from \cite{MR2965421} does not imply the existence of a precipitous ideal on $\omega_1$.

\end{abstract}

\maketitle


\section{Introduction}

Foreman-Magidor-Shelah~\cite{MR924672} proved the consistency of \emph{Martin's Maximum (MM)}, and isolated an interesting consequence:
\[
\dagger:  \ \ \text{Every poset which preserves stationary subsets of } \omega_1 \text{ is semiproper.}
\]
In fact they showed that MM implies generalized stationary set reflection, which in turn implied $\dagger$.  They proved that $\dagger$ implies precipitousness of the nonstationary ideal on $\omega_1$; thus $\dagger$ has large cardinal strength.   They also proved that generalized stationary set reflection implies presaturation of the nonstationary ideal on $\omega_1$; recently Usuba~\cite{MR3248209} reduced the assumption to ``$\dagger$ holds for posets of size $\le 2^{\omega_1}$".  He also proved that this bounded dagger principle implies a version of Chang's Conjecture.

The $\dagger$ principle is also interesting for particular posets definable in ZFC.  For example, it is a theorem of ZFC that Namba forcing preserves stationary subsets of $\omega_1$ (and even stronger properties, by \cite{MR1623206}).  Moreover:\footnote{Shelah~\cite{MR1623206} Chapter XII proves that semiproperness of Namba forcing implies $\text{SCC}^{\text{cof}}$; and a minor variation in the proof of Section 3 of Doebler~\cite{MR3065118} proves that $\text{SCC}^{\text{cof}}$ implies semiproperness of Namba forcing.}
\begin{theorem}[Shelah~\cite{MR1623206}; see also Section 3 of Doebler~\cite{MR3065118}]\label{thm_ShelahThm}
Semiproperness of Namba forcing is equivalent to a certain version of Strong Chang's Conjecture (the version we call $\text{SCC}^{\text{cof}}$ in Section \ref{sec_Prelims}).
\end{theorem}

This paper is about the $\dagger$ principle for the poset which adds a Cohen real, then shoots a continuous $\subset$-chain of length $\omega_1$ through $[\omega_2]^\omega - V$ using countable conditions, which we'll denote by
\begin{equation}\label{eq_GitikPoset}
\text{Add}(\omega) * \dot{\mathbb{C}}\big([\omega_2]^\omega - V \big).
\end{equation}
This poset has appeared in several applications in the literature, such as separating internal unboundedness from internal stationarity (Krueger~\cite{MR2674000}) and for applications involving thin stationary sets and disjoint club sequences (Friedman-Krueger~\cite{MR2276627}).  It always preserves stationary subsets of $\omega_1$, which follows from the following very useful fact:
\begin{fact}\label{fact_Gitik}[Abraham-Shelah~\cite{MR716625}, Gitik~\cite{MR820120}, Velickovic~\cite{MR1174395}]
If $\sigma$ is $\text{Add}(\omega)$-generic over $V$, then $V[\sigma]$ believes that $[\omega_2]^\omega - V$ is projective stationary;\footnote{That is, for every stationary $S \subseteq \omega_1$ there are stationarily many $z \in [\omega_2]^\omega - V$ such that $z \cap \omega_1 \in S$.  The Gitik and Velickovic arguments actually prove something much more general:  if $W$ is an outer model of $V$ and $W$ has some real that is not in $V$, then for every $W$-regular $\kappa \ge \omega_2^W$, $W$ believes that $[\kappa]^\omega - V$ is projective stationary.}
\end{fact}
In fact Friedman-Krueger~\cite{MR2276627} proved that it always satisfies a stronger (and RCS-iterable) condition of Shelah which is intermediate between ``preserves stationary subsets of $\omega_1$" and ``semiproper".  They asked:
\begin{question}[Question 1 of Friedman-Krueger~\cite{MR2276627}]\label{q_F_K_question}
Assuming Martin's Maximum, the poset $\text{Add}(\omega) * \dot{\mathbb{C}} \big( [\omega_2]^\omega - V  \big) $ is semiproper.  Is this poset semiproper in general?
\end{question}

We give a strong negative answer to Question \ref{q_F_K_question}, which we now describe.  Foreman-Magidor~\cite{MR1359154} defined a poset to be \emph{reasonable} iff it preserves the stationarity of $\big( [\theta]^\omega \big)^V$ for all $\theta \ge \omega_1$; this is a weak version of proper forcings (which are required to preserve \emph{all} stationary subsets of $[\theta]^\omega$).  Intuitively, a nonreasonable poset is as non-proper as possible while (possibly) preserving $\omega_1$; it kills the stationarity of the former club $[\theta]^\omega$ for some $\theta$.  In the following results it will be useful to stratify the notion of reasonableness; let us say that a poset is \emph{reasonable at $[\theta]^\omega$} if it preserves the stationarity of $\big( [\theta]^\omega \big)^V$, and \emph{nonreasonable at $[\theta]^\omega$} otherwise.  So in particular, the poset from Question \ref{q_F_K_question} is nonreasonable at $[\omega_2]^\omega$.  Notice that any $\omega_1$-preserving poset is reasonable at $[\omega_1]^\omega$; \textbf{so for $\omega_1$-preserving posets, the strongest possible degree of nonreasonableness is to be nonreasonable at $[\omega_2]^\omega$}.  Namba forcing is always nonreasonable at $[\omega_2]^\omega$, and so is the poset \eqref{eq_GitikPoset}; however the latter preserves all uncountable cofinalities because $\dot{\mathbb{C}} \big( [\omega_2]^\omega - V  \big)$ is forced by $\text{Add}(\omega)$ to be $\sigma$-distributive.

We prove that semiproperness of the poset from the Friedman-Krueger question, and semiproperness of nonreasonable posets in general, are closely related to strong versions of Chang's Conjecture.  The definitions of $\text{SCC}$, $\text{SCC}^{\text{cof}}$, and $\text{SCC}^{\text{cof}}_{\text{gap}}$ are given in Section \ref{sec_Prelims}.  

\begin{theorem}\label{thm_MainTheorem}
If there exists a semiproper poset which is nonreasonable at $[\omega_2]^\omega$, then Strong Chang's Conjecture (SCC) holds.
\end{theorem}

\begin{theorem}\label{thm_SCC_star_implies}
The principle $\text{SCC}^{\text{cof}}_{\text{gap}}$ implies that
\[
\text{Add}(\omega) * \dot{\mathbb{C}} \big( [\omega_2]^\omega - V  \big) 
\]
 is semiproper.
\end{theorem}

We also prove the following theorem, which is a minor modification of an argument of Sakai~\cite{MR2191239}:
\begin{theorem}[after \cite{MR2191239}]\label{thm_GeneralizeSakai}
Assume there exists a normal ideal $\mathcal{J}$ on $\omega_2$ such that $\wp(\omega_2)/\mathcal{J}$ is a proper forcing.  Then $\text{SCC}^{\text{cof}}_{\text{gap}}$ holds.
\end{theorem}

Now by Jech-Magidor-Mitchell-Prikry~\cite{MR560220}, an ideal satisfying the hypothesis of Theorem \ref{thm_GeneralizeSakai} can be forced from a measurable cardinal.\footnote{In fact one can arrange that the quotient is forcing equivalent to a $\sigma$-closed poset; and moreover the ideal can consistently be the nonstationary ideal on $\omega_2$ restricted to ordinals of uncountable cofinality.} Moreover SCC implies Chang's Conjecture (CC) which is equiconsistent with an $\omega_1$-Erd\H{o}s
 cardinal.    So the results above have the following corollary:

\begin{corollary}
\begin{eqnarray*}
& \text{CON}\Big(ZFC \ + \ \text{there is a measurable cardinal}\Big) \\
\implies &  \text{CON}\Big(ZFC \ + \ \text{ the poset }\text{Add}(\omega) * \dot{\mathbb{C}}\big( [\omega_2]^\omega - V \big) \text{ is semiproper} \Big) \\
\implies & \text{CON}\Big(ZFC \ + \ \text{there is an } \omega_1 \text{-Erd\H{o}s
 cardinal} \Big). \\
\end{eqnarray*}
\end{corollary}

We can also draw another corollary from Theorem \ref{thm_GeneralizeSakai} and core model theory.   By Foreman-Magidor-Shelah~\cite{MR924672}, the $\dagger$ principle implies that $\text{NS}_{\omega_1}$ is precipitous; and $\dagger$ implies semiproperness of Namba forcing which in turn (by Shelah's Theorem \ref{thm_ShelahThm}) is equivalent to $\text{SCC}^{\text{cof}}$.  In light of these facts, a natural question is whether $\text{SCC}^{\text{cof}}$ implies precipitousness of $\text{NS}_{\omega_1}$.  It does not; not even the stronger $\text{SCC}^{\text{cof}}_{\text{gap}}$ implies there is a precipitous ideal on $\omega_1$: 
\begin{corollary}\label{cor_gaps_cons_NoPrecip}
The principle $\text{SCC}^{\text{cof}}_{\text{gap}}$ (the strongest of the Chang's Conjecture variations considered in this paper) does \textbf{not} imply that there is a precipitous ideal on $\omega_1$.
\end{corollary}

Section \ref{sec_Prelims} provides the relevant background.  Section \ref{sec_MM} examines the relationship between Martin's Maximum, $\dagger$, and the principle $\text{SCC}^{\text{cof}}_{\text{gap}}$.  Section \ref{sec_LowerBound} proves Theorem \ref{thm_MainTheorem}; in fact a stronger theorem is proved there.  Section \ref{sec_UpperBound} proves Theorems \ref{thm_SCC_star_implies} and \ref{thm_GeneralizeSakai}.  Section \ref{sec_Cor_NS_omega_1} proves Corollary \ref{cor_gaps_cons_NoPrecip}.  Section \ref{sec_ConcludingRemarks} provides some concluding remarks about the relationship between the various strong Chang's Conjectures and special Aronszajn trees on $\omega_2$, and the relationship between bounded dagger principles and semiproperness of the poset \eqref{eq_GitikPoset}.

\section{Preliminaries}\label{sec_Prelims}

If $M$ and $N$ are sets which have transitive intersection with $\omega_1$, we write $M \sqsubseteq N$ to mean that $M \subseteq N$ and $M \cap \omega_1 = N \cap \omega_1$.  A poset $\mathbb{Q}$ is \emph{semiproper} iff for all sufficiently large $\theta$ and club-many (equivalently, every) countable $M \prec (H_\theta,\in,\mathbb{Q})$ and every $q \in M \cap \mathbb{Q}$ there is a $q' \le q$ such that
\[
q' \Vdash \check{M} \sqsubseteq \check{M}[\dot{G}].
\]

We frequently use the following fact (see e.g.\ Larson-Shelah~\cite{MR2030084}):
\begin{fact}\label{fact_Hulls}
If $\theta$ is regular uncountable, $\mathfrak{A}$ is a structure on $H_\theta$ in a countable language which has definable Skolem functions, $M \prec \mathfrak{A}$, and $Y$ is a subset of some $\eta \in M$, then 
\[
\text{Sk}^{\mathfrak{A}}(M \cup Y) = \{ f(y) \ | \  y \in [Y]^{<\omega} \text{ and }  f \in M \cap {}^{[\eta]^{<\omega}} H_\theta \}.
\]
\end{fact}

The classic Chang's Conjecture, which we will abbreviate by CC, has many equivalent formulations.\footnote{CC is often  expressed by $(\omega_2, \omega_1) \twoheadrightarrow (\omega_1, \omega)$}  One version states:  for every $\theta \ge \omega_2$ and every algebra $\mathfrak{A}$ on $H_\theta$, there is an $X \prec \mathfrak{A}$ with $|X \cap \omega_2| \ge \omega_1$ and $X \cap \omega_1 \in \omega_1$.

We will refer to several strengthenings of Chang's Conjecture.  \textbf{We caution the reader that the notation for various strengthenings of CC is very inconsistent across the literature}.  For example:
\begin{itemize}
 \item The notation $\text{CC}^*$ is used in the literature to refer to at least four distinct concepts (which are not known to be equivalent, as far as the author is aware).   The $\text{CC}^*$ from Todorcevic-Torres Perez~\cite{MR2965421} is what we are calling $\text{SCC}^{\text{cof}}_{\text{gap}}$, whereas the apparently weaker $\text{CC}^*$ from Usuba~\cite{MR3248209} and Torres Perez-Wu~\cite{MR3431031} is what we are calling $\text{SCC}^{\text{cof}}$.  The $\text{CC}^*$ from Todorcevic~\cite{MR1261218} is what we are calling $\text{SCC}$.  The $\text{CC}^*$ from Doebler-Schindler~\cite{MR2576698} is yet another  version which is much stronger and will not be considered here.\footnote{ Doebler-Schindler~\cite{MR2576698} proved that their version implies $\dagger$, which by Usuba~\cite{MR3248209} implies presaturation of $\text{NS}_{\omega_1}$.  Thus by Steel~\cite{MR1480175} and Jensen-Steel~\cite{MR3135495}, the Doebler-Schindler version of $\text{CC}^*$ has consistency strength at least a Woodin cardinal; whereas all the versions of Chang's Conjecture considered in this paper can be forced from a measurable cardinal.} 
 \item ``Strong Chang's Conjecture" from Woodin~\cite{MR2723878} is not the same as ``Strong Chang's Conjecture" from Sharpe-Welch~\cite{MR2817562} (Woodin's is what we call $\text{SCC}$, and Sharpe-Welch's is what we call $\text{SCC}^{\text{cof}}$).  A similar discrepancy appears in the use of the notation $\text{CC}^+$ in \cite{MR2723878} and \cite{MR2817562}, though we will not deal with either of these versions.   The ``Strong Chang's Conjecture" of Foreman-Magidor-Shelah~\cite{MR924672} is apparently weaker than the ``Strong Chang's Conjecture" of Woodin~\cite{MR2723878}.
\end{itemize}

Table \ref{table_Translate_SCC} provides a translation for the various uses in the literature.

\begin{definition}\label{def_SCC} 
\emph{Strong Chang's Conjecture} (SCC) is the statement:  for all sufficiently large regular $\theta$, all wellorders $\Delta$ of $H_\theta$, and all countable $M \prec (H_\theta,\in,\Delta)$, there exists a $\widetilde{M}$ such that 
\begin{itemize}
 \item $\widetilde{M}  \prec (H_\theta,\in,\Delta)$; 
 \item $M \sqsubset \widetilde{M}$; and
 \item $\widetilde{M} \cap [\text{sup}(M \cap \omega_2), \omega_2) \ne \emptyset$.
\end{itemize}
\end{definition}

SCC implies CC.  In fact, SCC is equivalent to saying that club-many $M \in [H_\theta]^\omega$ can be $\sqsubset$-extended to a model whose intersection with $\omega_2$ is uncountable; whereas CC is equivalent to this holding for just stationarily many $M \in [H_\theta]^\omega$.   SCC is strictly stronger than CC because SCC implies $2^\omega \le \omega_2$, whereas CC places no bound on the continuum (see Section 2 of Todorcevic~\cite{MR1261218}).

We will use even further strengthenings of SCC.  The following requires that one can not only obtain proper end-extensions (as SCC requires), but that an end-extension with arbitrarily large supremum below $\omega_2$ can be found:

\begin{definition}\label{def_CC_star}
$\text{SCC}^{\text{cof}}$ is the statement: for all sufficiently large regular $\theta$ and every wellorder $\Delta$ of $H_\theta$ and every countable $N \prec (H_\theta,\in,\Delta)$, there are cofinally many $\alpha \in \omega_2$ such that there exists an $N' \prec (H_\theta,\in, \Delta \})$ where:
\begin{enumerate}
 \item $N \sqsubseteq N'$;
 \item $N' \cap [\alpha,\omega_2) \ne \emptyset$.
\end{enumerate}

\end{definition}

Finally, the strongest version we will encounter is the following, which requires arbitrarily large gaps above the model to be $\sqsubset$-extended:
\begin{definition}
$\text{SCC}^{\text{cof}}_{\text{gap}}$ is defined exactly the same as $\text{SCC}^{\text{cof}}$, except the following additional requirement is placed on the $N'$ from Definition \ref{def_CC_star}:
\[ 
N' \cap [ \text{sup}(N\cap \omega_2), \alpha) = \emptyset.
\]
\end{definition}

The additional requirement for $\text{SCC}^{\text{cof}}_{\text{gap}}$ will be important in the proof of Theorem \ref{thm_SCC_star_implies}.  The following implications are straightforward:
\[
\text{SCC}^{\text{cof}}_{\text{gap}} \implies \text{SCC}^{\text{cof}} \implies \text{SCC} \implies \text{CC}
\]


\begin{table}
\caption{Translating strong versions of Chang's Conjecture}\label{table_Translate_SCC}
\scalebox{0.7}{
\begin{tabular}{ |l|l||c|c|c| }
\hline
\multicolumn{2}{|c||}{Elsewhere in the literature} & \multicolumn{3}{|c|}{Corresponds to our:}\\
\hline
Source & Their notation & SCC & $\text{SCC}^{\text{cof}}$ & $\text{SCC}^{\text{cof}}_{\text{gap}}$ \\
\hline
\hline
Todorcevic~\cite{MR1261218} & $\text{CC}^*$ & \checkmark  & &  \\
Todorcevic-Torres Perez~\cite{MR2965421} & $\text{CC}^*$ &   & & \checkmark \\
Usuba~\cite{MR3248209} & $\text{CC}^*$ & \checkmark  &  &  \\
Usuba~\cite{MR3248209} & $\text{CC}^{**}$ &   & \checkmark &  \\
Torres Perez-Wu~\cite{MR3431031} & $\text{CC}^*$ &   & \checkmark &  \\
Doebler~\cite{MR3065118} & $\text{CC}^*$ & \checkmark & & \\
Shelah~\cite{MR1623206} & version in XII Theorem 2.5 &   & \checkmark &  \\
Sharpe-Welch~\cite{MR2817562} & SCC &   & \checkmark &  \\
Woodin~\cite{MR2723878}& SCC (Def 9.101 part 2) & \checkmark  & & \\
\hline
\end{tabular}
}
\end{table}


The following lemma is standard and streamlines arguments involving variants of SCC, by allowing one to replace ``every" by ``club-many", but \emph{without} having to strengthen the algebra in which the end extensions are required to be elementary.  
\begin{lemma}\label{lem_WoodinLemma}
  The following are equivalent:
\begin{enumerate}
 \item\label{item_SCC} $\text{SCC}^{\text{cof}}_{\text{gap}}$ (as in Definition \ref{def_CC_star});
 \item\label{item_ApparentlyWeaker} There are club-many $N \in [H_{\omega_3}]^\omega$ such that for cofinally many $\alpha \in \omega_2$, there exists an $N' \prec (H_{\omega_3},\in)$ where:
\begin{enumerate}
 \item $N \sqsubseteq N'$;
 \item $N' \cap [\alpha,\omega_2) \ne \emptyset$; and
 \item $N' \cap [ \text{sup}(N\cap \omega_2), \alpha) = \emptyset$.
\end{enumerate}
\end{enumerate}
\end{lemma}
Lemma \ref{lem_WoodinLemma} is similar to Lemma 9.103 of \cite{MR2723878}; however since there are some confusing typos in the ``3 implies 1" direction of the latter, we provide a short proof.\footnote{Lemma 9.103 of \cite{MR2723878} is the version for SCC;  Lemma \ref{lem_WoodinLemma} above is the version for $\text{SCC}^{\text{cof}}_{\text{gap}}$.}
\begin{proof}
That \ref{item_SCC} implies \ref{item_ApparentlyWeaker} is trivial.  For the other direction, fix a regular $\theta > |H_{\omega_3}|$ and a wellorder $\Delta$ on $H_\theta$.  Fix a countable $\widetilde{N} \prec (H_\theta,\in,\Delta)$.  The assumptions and the elementarity of $\widetilde{N}$ imply there is some algebra $\mathfrak{A}$ on $H_{\omega_3}$ such that $\mathfrak{A} \in \widetilde{N}$ and $\mathfrak{A}$ has the properties listed in \ref{item_ApparentlyWeaker}.  In particular since $\mathfrak{A} \in \widetilde{N}$ then $N:=\widetilde{N} \cap H_{\omega_3} \prec \mathfrak{A}$; so there are cofinally many $\alpha<\omega_2$ with the properties listed in \ref{item_ApparentlyWeaker}.  Fix such an $\alpha$ and an $N' \sqsupseteq N$ such that $N' \prec (H_{\omega_3},\in)$ and $N'$ has the other properties listed in \ref{item_ApparentlyWeaker}.  Define
\[
\widetilde{N}':= \{ f(y) \ | \  f \in \widetilde{N} \cap {}^{\omega_2} H_\theta  \text{ and } y \in N' \cap \omega_2 \} .
\]
Since $\widetilde{N} \prec (H_\theta,\in,\Delta)$ and $(H_\theta,\in,\Delta)$ has definable Skolem functions, then Fact \ref{fact_Hulls} implies that $\widetilde{N}' \prec (H_\theta,\in,\Delta)$.  Clearly $\widetilde{N} \subset \widetilde{N}'$. Furthermore if $f(y) < \omega_2$ where $f \in \widetilde{N}$ and $y \in N'$ then without loss of generality $f: \omega_2 \to \omega_2$; so $f \in \widetilde{N} \cap H_{\omega_3} = N \subset N'$.  Since $y$ and $f$ are both in $N'$ then $f(y) \in N'$.  This shows that $\widetilde{N}' \cap \omega_2 = N' \cap \omega_2$, and it follows that $\widetilde{N}'$ has the desired properties with respect to $\widetilde{N}$.
\end{proof}

The following lemma is very similar to Lemma \ref{lem_WoodinLemma}, so we omit the proof:
\begin{lemma}\label{lem_WoodinLemma_CC_club}
  The following are equivalent:
\begin{enumerate}
 \item\label{item_CC_club} Strong Chang's Conjecture  (Definition \ref{def_SCC})
 \item There are club-many $N \in [H_{\omega_3}]^\omega$ such that there exists an $N' \prec (H_{\omega_3},\in)$ where:
\begin{enumerate}
 \item $N \sqsubset N'$;
 \item $N' \cap [\text{sup}(N \cap \omega_2), \omega_2) \ne \emptyset$.
\end{enumerate}
\end{enumerate}

\end{lemma}

The following lemma basically says that if $M \sqsubseteq N$ and they have access to the same wellorder of $H_{\omega_2}$, then $N \cap \omega_2$ is an end extension of $M \cap \omega_2$.\footnote{Lemma \ref{lem_EndExtend} is the reason that our Definition \ref{def_SCC} of SCC is equivalent to part 2 of Definition 9.101 of \cite{MR2723878}.}   
\begin{lemma}\label{lem_EndExtend}
Suppose $w$ is a wellorder on $H_{\omega_2}$ and $M$ is a countable elementary substructure of $(H_{\omega_2},\in,w)$.  Suppose $N$ is another countable model, perhaps in some outer model of $V$, such that $M \sqsubseteq N$ and $N \cap H_{\omega_2}^V \prec (H_{\omega_2}^V,\in,w)$.  Then $N \cap \omega^V_2$ is an end-extension of $M \cap \omega^V_2$; i.e.\ \[ N \cap \text{sup}(M \cap \omega^V_2) = M \cap \omega^V_2. \]
\end{lemma}
\begin{proof}
One direction is trivial, since $M \subset N$ by assumption.  For the other direction, let $\zeta \in N \cap \text{sup}(M \cap \omega^V_2)$.  Since $\text{sup}(M \cap \omega^V_2)$ is a limit ordinal there is some $\beta \in M \cap \omega_2^V$ such that $\zeta < \beta$.  Let $f$ be the $w$-least bijection from $\omega_1 \to \beta$.  Since $\beta \in M \subset N$, $M \prec (H_{\omega_2}^V,\in,w)$, and $N \cap H_{\omega_2}^V \prec (H_{\omega_2}^V,\in,w)$, then $f \in M \cap N$.  Then \[ \zeta \in N \cap \beta = f[N \cap \omega_1] = f[M \cap \omega_1] = M \cap \beta. \]
\end{proof}

Finally we recall a standard fact:
\begin{fact}\label{fact_includeMasterCondition}
If $\mathbb{P}$ is a proper poset, $S$ is a stationary subset of $[H_\theta]^\omega$ for some $\theta \ge 2^{|\mathbb{P}|}$, and $G$ is generic for $\mathbb{P}$, then $V[G]$ believes that there are stationarily many $N \in S$ such that $G$ includes a master condition for $N$.
\end{fact}
\begin{proof}
If not then there is some condition $p$ and some name $\dot{\mathfrak{A}}$ for an algebra on $H_\theta^V$ such that
\begin{equation}\label{eq_ForcesNoMC}
p \Vdash \forall N \in \check{S} \ \ N \prec \dot{\mathfrak{A}} \implies \dot{G} \text{ does not include a master condition for } N.
\end{equation}
The stationarity of $S$ ensures that there is some $N \in S$ such that $p \in N$ and $N = \widetilde{N} \cap H_\theta$ for some 
\[
\widetilde{N} \prec (H_{|H_\theta|^+},\in, \dot{\mathfrak{A}}).
\]
Since $\mathbb{P}$ is proper and $p \in \widetilde{N}$ then there is a $p' \le p$ which is a master condition for $\widetilde{N}$.  Since $\dot{\mathfrak{A}} \in \widetilde{N}$ and $p'$ is a master condition for $\widetilde{N}$ then
\begin{equation*}
p' \Vdash \widetilde{N} \cap H_\theta^V = N \prec \dot{\mathfrak{A}}, \ N \in S, \ \text{ and } \dot{G} \text{ includes a master condition for } N .
\end{equation*}

Since $p' \le p$, this contradicts \eqref{eq_ForcesNoMC}.
\end{proof}

\section{Martin's Maximum, $\dagger$, and $\text{SCC}^{\text{cof}}_{\text{gap}}$}\label{sec_MM}

Recall from the introduction that Martin's Maximum (MM) implies $\dagger$, which in turn implies $\text{SCC}^{\text{cof}}$.  In this section we show that MM implies the principle $\text{SCC}^{\text{cof}}_{\text{gap}}$ introduced in Section \ref{sec_Prelims}, while $\dagger$ does not.

If $\Gamma$ is a subclass of $\{ W \ : \ |W|=\omega_1 \subset W \}$, $\text{RP}_\Gamma$ abbreviates the statement:  For every regular $\theta \ge \omega_2$ and every stationary $S \subseteq [H_\theta]^\omega$, there exists a $W \in \Gamma \cap [H_\theta]^{\omega_1}$ such that $S \cap [W]^\omega$ is stationary.  The following is a standard fact:
\begin{fact}\label{fact_StatManyReflect}
If $\text{RP}_\Gamma$ holds, then for every $\theta \ge \omega_2$ and every stationary $S \subseteq [H_\theta]^\omega$ there are in fact stationarily many $W \in [H_\theta]^{\omega_1} \cap  \Gamma$ such that $S \cap [W]^\omega$ is stationary.
\end{fact}
\begin{proof}
If not, there is a function $F: [H_\theta]^{<\omega} \to H_\theta$ and a stationary $S \subseteq [H_\theta]^\omega$ such that $S \cap [W]^\omega$ is nonstationary for every $W \in \Gamma$ that is closed under $F$.  Let $S':= \{ M \in S \ : \ M \text{ is closed under } F \}$.  Then $S'$ is stationary, so by assumption there is a $W \in \Gamma$ such that $S' \cap [W]^\omega$ is stationary (and hence $S \cap [W]^\omega$ is also stationary).  If $p \in [H_\theta]^{<\omega} \cap W$ there is some $M \in S'$ such that $p \in M$, and since $M \in S'$ we have $F(p) \in M \subset W$.  So $W$ is closed under $F$, a contradiction.  
\end{proof}

Two particular subclasses of $\{ W \ : \ |W|=\omega_1 \subset W \}$ are relevent in what follows.  $\text{IA}$ denotes the class of $W$ such that $\omega_1 \subset W$ and there is some $\subseteq$-increasing, continuous sequence $\langle N_\xi \ : \ \xi < \omega_1 \rangle$ of countable sets such that $W= \bigcup_{\xi < \omega_1} N_\xi$ and $\vec{N} \restriction \xi \in W$ for every $\xi < \omega_1$ (the IA stands for ``internally approachable").  $\text{IC}$ denotes the class of $W$ such that $|W|=\omega_1 \subset W$ and $W \cap [W]^\omega$ contains a club in $[W]^\omega$ (the IC stands for ``internally club", as introduced in Foreman-Todorcevic~\cite{MR2115072}).

\begin{lemma}\label{lem_RPIC_SCCcofgap}
$\text{RP}_{\text{IC}}$ implies $\text{SCC}^{\text{cof}}_{\text{gap}}$.
\end{lemma}
\begin{proof}
Suppose not.  By Lemma \ref{lem_WoodinLemma} there is a stationary $S \subseteq [H_{\omega_3}]^\omega$ such that for every $M \in S$, there is a $\beta_M < \omega_2$ such that for all $\beta \in [\beta_M,\omega_2)$, there is \textbf{no} countable $N \prec (H_{\omega_3},\in)$ such that $M \sqsubset N$, $N \cap \big( \text{sup}(M \cap \omega_2), \beta \big) = \emptyset$, and $N \cap [\beta,\omega_2) \ne \emptyset$.  By Fact \ref{fact_StatManyReflect} there exists a $W \in \text{IC} \cap [H_{\omega_3}]^{\omega_1}$ such that $S \cap [W]^\omega$ is stationary, and $W \prec (H_{\omega_3},\in,S, P, \Delta)$ where $\Delta$ is a wellordering of $H_{\omega_3}$ and $P$ is the predicate $\{ (M,\beta_M) \ : \  M \in S \}$.  Since $W \in \text{IC}$, $S \cap W \cap [W]^\omega$ is stationary.  It follows by normality and $\sigma$-completeness of the nonstationary ideal that there is some $M \in S \cap W \cap [W]^\omega$ such that 
\[
N \cap W = M \text{, where } N:=\text{Sk}^{(H_{\omega_3},\in,\Delta)}\big( M \cup \{ W \cap \omega_2 \} \big).
\]
Then $M \sqsubset N$, $N \cap \big( \text{sup}(M \cap \omega_2), W \cap \omega_2 \big) = \emptyset$, and $W \cap \omega_2 \in N$; it follows that $\beta_M \ge W \cap \omega_2$.  On the other hand, since $M \in W$ and $W$ is elementary with respect to the predicate $P$, $\beta_M < W \cap \omega_2$.  Contradiction.
\end{proof}

\begin{corollary}
Martin's Maximum implies $\text{SCC}^{\text{cof}}_{\text{gap}}$.
\end{corollary}
\begin{proof}
MM implies $\text{RP}_{\text{IA}}$ (see \cite{MR1903851} and \cite{MR924672}). Clearly $\text{IA} \subseteq \text{IC}$, and so $\text{RP}_{\text{IA}} \implies \ \text{RP}_{\text{IC}}$.  The corollary then follows from Lemma \ref{lem_RPIC_SCCcofgap}.
\end{proof}

Since $\dagger \implies \text{SCC}^{\text{cof}}$ by Shelah's Theorem \ref{thm_ShelahThm}, it is natural to ask if $\dagger$ also implies $\text{SCC}^{\text{cof}}_{\text{gap}}$.  It does not.  To see this we use a result of Usuba~\cite{MR3523539}.  For $m < n$ let $S^n_m$ denote the set $\omega_n \cap \text{cof}(\omega_m)$.  A sequence $\vec{d} = \langle d_\alpha \ : \ \alpha \in S^2_0 \rangle$ is called a \textbf{nonreflecting ladder system for $\boldsymbol{S^2_0}$} iff each $d_\alpha$ is a cofinal subset of $\alpha$ of ordertype $\omega$, and for every $\gamma \in S^2_1$ there exists a club $D \subseteq \gamma$ and an injective function $f: D \to \text{ORD}$ such that $f(\alpha) \in d_\alpha$ for all $\alpha \in D$.  

\begin{lemma}\label{lem_NRladder}
Suppose there is a nonreflecting ladder system for $S^2_0$.  Then $\text{SCC}^{\text{cof}}_{\text{gap}}$ fails.
\end{lemma}
\begin{proof}
Let $\Delta$ be a wellordering of $H_{\omega_3}$, and let $\vec{d}$ be the $\Delta$-least nonreflecting ladder system for $S^2_0$.  Let $S$ be the set of $M \in [H_{\omega_3}]^\omega$ such that $M \prec \mathfrak{A}:=(H_{\omega_3},\in,\Delta)$ and $M \supset d_{\text{sup}(M \cap \omega_2)}$; $S$ is easily seen to be stationary.\footnote{In fact it is stationary and costationary, as shown in Usuba~\cite{MR3523539}.}.  Fix $M \in S$.  We prove that if $N$ is any countable elementary substructure of $\mathfrak{A}$ such that $M \subseteq N$ and 
\[N \cap \big[ \text{sup}(M \cap \omega_2), \omega_2 \big) \ne \emptyset,
\]
then $\text{sup}(M \cap \omega_2) \in N$; since $S$ is stationary this will imply that $\text{SCC}^{\text{cof}}_{\text{gap}}$ fails.\footnote{The ``cofinal" requirement of $\text{SCC}^{\text{cof}}_{\text{gap}}$ isn't used here, just the ``gap" requirement.  That is, the proof actually shows that if there is a nonreflecting ladder system for $S^2_0$, then there are stationarily many models $M$ for which there is no $\beta \in \big( \text{sup}(M \cap \omega_2),\omega_2  \big)$ such that $\text{Sk}^{\mathfrak{A}}(M \cup \{ \beta \}) \cap \beta = M \cap \omega_2$.}  So fix such an $N$, and let $\gamma$ be the least member of $N \cap \big[ \text{sup}(M \cap \omega_2), \omega_2 \big)$.  Suppose toward a contradiction that $\gamma > \text{sup}(M \cap \omega_2)$; then $\gamma$ must have cofinality $\omega_1$.  Since $\vec{d}$ is nonreflecting and $N \prec \mathfrak{A}$, in $N$ there is a club $D \subset \gamma$ and an injective $f: D \to \text{ORD}$ such that $f(\alpha) \in d_\alpha$ for every $\alpha \in D$.  By minimality of $\gamma$ and the facts that $D \in N$ and $D$ is unbounded in $\gamma$, $\text{sup}(M \cap \omega_2)$ is a limit point of $D$, and hence an element of $D$ because $D$ is closed.  Now $f \big(\text{sup}(M \cap \omega_2) \big) \in d_{\text{sup}(M \cap \omega_2)} \subset M$ because $M \in S$, and hence $f \big(\text{sup}(M \cap \omega_2) \big) \in N$.  But then the injectivity of $f$, and the fact that $f \in N$, ensure that $\text{sup}(M \cap \omega_2) \in N$, a contradiction.
\end{proof}

Section 6 of Usuba~\cite{MR3523539} produces a model where $\dagger$ holds\footnote{He shows that the model satisfies ``Semistationary set reflection", which is equivalent to $\dagger$.} and there exists a nonreflecting ladder system for $S^2_0$.  Together with Lemma \ref{lem_NRladder} this yields a model witnessing the following corollary.

\begin{corollary}
The $\dagger$ principle does not imply $\text{SCC}^{\text{cof}}_{\text{gap}}$.
\end{corollary}

\section{Proof of Theorem \ref{thm_MainTheorem}}\label{sec_LowerBound}

If $H \supseteq \omega_1$ and $S \subseteq [H]^\omega$, we say that $S$ is \emph{semistationary} iff 
\[
\{ N \in [H]^\omega \ | \  \exists M \in S \  M \sqsubseteq N  \} \text{ is stationary}.
\]

Clearly every stationary set is semistationary, but the converse is false.  Just as properness is equivalent to preservation of stationary sets, Shelah~\cite{MR1623206} shows that semiproperness of a poset $\mathbb{Q}$ is equivalent to:  every semistationary set in $V$ remains semistationary in $V^{\mathbb{Q}}$.  This, in turn, is easily equivalent to saying that every stationary set in $V$ remains at least semistationary in $V^{\mathbb{Q}}$.  We prove the following theorem, which is slightly more general than Theorem \ref{thm_MainTheorem} because it deals with arbitrary semistationary preserving outer models, rather than just forcing extensions.

\begin{theorem}\label{thm_MainTheorem_Generalized}
Assume $V \subset W$ are models of ZFC, every stationary set of countable models in $V$ remains \textbf{semi}stationary in $W$, and $\big([\omega_2]^\omega \big)^V$ is nonstationary in $W$.  Then 
\[
V \models \ \text{SCC}.
\]
\end{theorem}
\begin{proof}
First note that the assumptions ensure that $V$ and $W$ have the same $\omega_1$; otherwise the stationary set $\big( [\omega_1]^\omega)^V$ in $V$ would fail to be semistationary in $W$.

Working in $V$, fix a regular $\theta \ge \omega_3$ and a wellorder $\Delta$ of $H_\theta$.  Let
\[ 
\mathfrak{A} = (H_\theta,\in,\Delta).
\]

By Lemma \ref{lem_WoodinLemma_CC_club} and Lemma \ref{lem_EndExtend} it suffices to prove that there are club-many $M \in [H_\theta]^\omega$ which can be $\sqsubset$-extended to an elementary substructure of $\mathfrak{A}$ which includes some ordinal in $\omega_2 - M$.  Suppose toward a contradiction that this fails; let $S$ denote the stationary collection of counterexamples, and without loss of generality assume $M \prec \mathfrak{A}$ for every $M \in S$.

The hypotheses of the theorem ensure that
\begin{equation}\label{eq_S_semistat}
W \models \ S \text{ is semistationary} \text{ in } [H^V_\theta]^\omega.
\end{equation}

Work in $W$.  Let $F:[\omega_2^V]^{<\omega} \to \omega_2^V$ witness that $\big([\omega_2]^\omega \big)^V$ is nonstationary in $W$; so
\begin{equation*}
W \models \ \forall z \in [\omega_2^V]^\omega \ \  \ z \text{ closed under } F \implies z \notin H^V_\theta.
\end{equation*}
(Here we use $H^V_\theta$, which is an element of $W$, because we are not necessarily assuming that $V$ is definable in $W$).   Let $\Omega > |H_\theta|$ be regular and define
\[ 
\mathfrak{B}:=( H^W_\Omega, \in, \{ \mathfrak{A}, F   \}  ).
\]

By \eqref{eq_S_semistat} and standard facts about liftings of stationary sets, in $W$ there is some countable $N \prec \mathfrak{B}$ such that $N \sqsupseteq M$ for some $M \in S$.  Since $F \in N$ then $N \cap \omega_2^V \notin V$; together with the facts that $M \in V$ and $M \subset N$ this implies
\[ 
M \cap \omega_2^V \subsetneq N \cap \omega_2^V.
\]

Pick some $\zeta \in \omega^V_2 \cap (N-M)$ and consider the following set, which is an element of $V$ (note that $\mathfrak{A}$ has definable Skolem functions):
\[
M':= \text{Sk}^{\mathfrak{A}}(M \cup \{ \zeta \}).
\]

Since $N \prec \mathfrak{B}$ (so $\mathfrak{A} \in N$) and $M \cup \{ \zeta \} \subset N$, then $M' \subseteq N$.  So in summary we have that $M' \prec \mathfrak{A}$, $M \subset M' \subseteq N$ and $M \sqsubset N$.  It follows that $M \sqsubset M' \prec \mathfrak{A}$.  This contradicts that $M \in S$.

\end{proof}

\begin{remark}
Recall that Shelah~\cite{MR1623206} proved that semiproperness of Namba forcing implies $\text{SCC}^{\text{cof}}$.  It is tempting to try to modify the proof of Theorem \ref{thm_MainTheorem_Generalized} above to achieve $\text{SCC}^{\text{cof}}$, rather than just SCC, as follows.  Instead of working with $M$, work instead with some maximal $\sqsubseteq$-extension of $M$ which is elementary in $\mathfrak{A}$.  The problem is that Theorem \ref{thm_MainTheorem_Generalized} implies that such a $\sqsubseteq$-maximal extension will be uncountable, and thus apparently irrelevant to the preservation of semistationary sets of countable models.
\end{remark}

\section{Proofs of Theorems \ref{thm_SCC_star_implies} and \ref{thm_GeneralizeSakai}}\label{sec_UpperBound}

Before proceeding to the proofs of Theorems \ref{thm_SCC_star_implies} and \ref{thm_GeneralizeSakai}, note that 
Sakai~\cite{MR2191239} proved that if there is a normal ideal $\mathcal{J}$ on $\omega_2$ such that $P(\omega_2)/\mathcal{J}$ is a semiproper poset, then $\text{SCC}^{\text{cof}}$ holds.  However it is not clear if $\text{SCC}^{\text{cof}}$ would suffice to prove Theorem \ref{thm_SCC_star_implies}; i.e.\ we seem to need $\text{SCC}^{\text{cof}}_{\text{gap}}$, not just $\text{SCC}^{\text{cof}}$, to prove semiproperness of the poset
\[
\text{Add}(\omega) * \dot{\mathbb{C}}\big([\omega_2]^\omega - V \big).
\]

\subsection{Proof of Theorem \ref{thm_SCC_star_implies}}

Let $\theta > |H_{\omega_2}|$.  Fix some $w \in H_{\omega_3}$ which is a wellorder of $H_{\omega_2}$.  Fix any $N \prec (H_\theta, \in)$ such that $w \in N$.  Since $w \in N$ then Lemma \ref{lem_EndExtend} implies:
\begin{equation}\label{eq_ExtensionsOf_N}
N \subseteq Q \sqsubseteq R \text{ and } Q,R \prec (H_\theta,\in) \ \implies \ \ R \cap \text{sup}(Q \cap \omega_2) = Q \cap \omega_2.
\end{equation}
Let $(p,\dot{f})$ be a condition in $N \cap \text{Add}(\omega) * \dot{\mathbb{C}}\big([\omega_2]^\omega - V \big)$; we want to find a semigeneric condition for $N$ below it.

Let $\sigma$ be $\big(V,\text{Add}(\omega)\big)$-generic with $p \in \sigma$ and $f:= \dot{f}_\sigma$.  
\begin{globalClaim}\label{clm_EndExtension}
There is some $M \prec (H_\theta[\sigma], \in)$ such that
\begin{itemize}
 \item $f \in M$;
 \item $M \cap \omega_2 \notin V$; and
 \item $N[\sigma] \sqsubseteq M$.
\end{itemize}
\end{globalClaim}

Note that this claim will finish the proof of Theorem \ref{thm_SCC_star_implies}, because there will then be some $f' \le f$ which is a totally generic condition for $\big( M, \mathbb{C}([\omega_2]^\omega - V) \big)$.  To construct such an $f'$ (assuming the claim holds), first define a descending chain $\langle f_n \ : \ n \in \omega \rangle$ with $f_0 = f$ such that the upward closure of $\{ f_n \ : \ n \in \omega \}$ is an $\big( M, \mathbb{C}([\omega_2]^\omega - V) \big)$-generic filter.  By Fact \ref{fact_Gitik}, $[\omega_2]^\omega - V$ is projective stationary and in particular unbounded in $([\omega_2]^\omega)^{V[\sigma]}$, so an easy density argument ensures that $\bigcup_{n \in \omega} \text{range}(f_n) = M \cap \omega_2$ and $\text{sup}_{n \in \omega} \text{dom}(f_n) = M \cap \omega_1$.  Then 
\[
\bigcup_{n \in \omega} f_n \cup \{ M \cap \omega_1 \mapsto  M \cap \omega_2  \}
\]
satisfies the continuity requirement, and is a condition because $M \cap \omega_2 \notin V$.  Since  $N[\sigma] \sqsubseteq M$ and $f'$ is a generic condition for $\big( M, \mathbb{C}([\omega_2]^\omega - V) \big)$, $f'$ is a semigeneric condition for $\big( N[\sigma], \mathbb{C}([\omega_2]^\omega - V) \big)$; and since $p$ is an $\big(N,\text{Add}(\omega) \big)$ master condition then $(p,\dot{f}')$ will be the semigeneric condition we seek (where $\dot{f}'$ is a name for $f'$).  

\begin{proof}
(of Claim \ref{clm_EndExtension})  The following coding argument in some ways resembles arguments from Gitik~\cite{MR820120} and Velickovic~\cite{MR1174395}.  In $V[\sigma]$ we recursively define three sequences of elementary substructures of $(H^V_\theta,\in)$:
\begin{eqnarray*}
\langle Q^{\mathcal{M}}_n \ | \ n < \omega \rangle \\
\langle Q^{\mathcal{Y}}_n \ | \ n < \omega \rangle \\
\langle Q^{\mathcal{N}}_n \ | \ n < \omega \rangle .
\end{eqnarray*}

Intuitively, the ``$\mathcal{M}$" (for Move) sequence will tell us when to move to the next decimal place; the ``$\mathcal{Y}$" (for Yes) sequence will indicate where to put a 1; and the ``$\mathcal{N}$" (for No) sequence will indicate where to put a 0.

Define a function 
\[
\text{Active}: \omega \to \{ \mathcal{M},\mathcal{Y},\mathcal{N}\}
\]
as follows.  If $n$ is even, say $n= 2k$, then $\text{Active}(n) = \mathcal{Y}$ if the $k$-th bit of $\sigma$ is 1, $\text{Active}(n) = \mathcal{N}$ if the $k$-th bit of $\sigma$ is 0.  If $n$ is odd, then $\text{Active}(n)$ is always $\mathcal{M}$.  For $\mathcal{X} \in \{ \mathcal{M},\mathcal{Y},\mathcal{N} \}$ and $n < \omega$ we say that $\mathcal{X}$ is active at stage $n$ if $\text{Active}(n) = \mathcal{X}$; otherwise $\mathcal{X}$ is passive at stage $n$. 

Set $Q^{\mathcal{X}}_0:= N$ for all $\mathcal{X} \in \{ \mathcal{M},\mathcal{Y},\mathcal{N} \}$.  Assume $Q^{\mathcal{X}}_n$ is defined for each $\mathcal{X} \in \{ \mathcal{M},\mathcal{Y},\mathcal{N} \}$.  Set $s^{\mathcal{X}}_n:= \text{sup}(Q^{\mathcal{X}}_n \cap \omega_2)$ for each $\mathcal{X} \in \{ \mathcal{M},\mathcal{Y},\mathcal{N} \}$, and $s_n:= \text{max}\{ s^{\mathcal{M}}_n, s^{\mathcal{Y}}_n, s^{\mathcal{N}}_n  \}$.  We then define the $n+1$-st models as follows: 
\begin{itemize}
 \item If $\mathcal{X}$ is passive at stage $n$ then $Q^{\mathcal{X}}_{n+1}:= Q^{\mathcal{X}}_n$.
 \item If $\mathcal{X}$ is active at stage $n$, then we use that $\text{SCC}^{\text{cof}}_{\text{gap}}$ holds in $V$ to find some $Q^{\mathcal{X}}_{n+1} \in V \cap [H^V_\theta]^\omega$ such that:
 \begin{itemize}
  \item $Q^{\mathcal{X}}_{n+1} \prec (H^V_\theta, \in)$;
  \item $Q^{\mathcal{X}}_n \sqsubset Q^{\mathcal{X}}_{n+1}$;
  \item $Q^{\mathcal{X}}_{n+1} \cap [s^{\mathcal{X}}_n, s_n ) = \emptyset$; and
  \item $Q^{\mathcal{X}}_{n+1} \cap [s_n,\omega_2) \ne \emptyset$. 
 \end{itemize}   
\end{itemize} 

This completes the recursive definition of the three sequences of models.  Note that \eqref{eq_ExtensionsOf_N} and the construction of the models implies that for all $n < \omega$ and each $\mathcal{X} \in \{\mathcal{M},\mathcal{Y},\mathcal{N}\}$:
\begin{equation}\label{End_Extend_active}
Q^{\mathcal{X}}_{n+1} \cap s_n = Q^{\mathcal{X}}_n \cap \omega_2.
\end{equation}

Let 
\[
Q^{\mathcal{X}}_\omega:= \bigcup_{n \in \omega} Q^{\mathcal{X}}_n 
\]
and set $z^{\mathcal{X}}:= Q^{\mathcal{X}}_\omega \cap \omega_2$ for each ${\mathcal{X}} \in \{ \mathcal{M},\mathcal{Y},\mathcal{N} \}$.  Also define a sequence $\vec{\alpha}$ by setting $\alpha_0:= \text{sup}(N \cap \omega_2)$, and for each $n \in \omega$, define $\alpha_{n+1}$ to be the least element of $Q^{\text{Active}(n)}_{n+1} \cap [s_n,\omega_2)$.

Notice that the construction of the sequences of models ensures:
\begin{enumerate}
 \item If $H$ is a transitive $ZF^-$ model and $z^{\mathcal{M}}$, $z^{\mathcal{Y}}$, and $z^{\mathcal{N}}$ are all elements of $H$, then $\vec{\alpha} \in H$, via the following algorithm.  Clearly $\alpha_0 \in H$ by transitivity.  Given $\alpha_n$, there is a \emph{unique} $\mathcal{X}_n \in \{  \mathcal{M},\mathcal{Y},\mathcal{N} \}$ such that $\alpha_n \in z^{\mathcal{X}_n}$;\footnote{Namely $\mathcal{X}_n = \text{Active}(n)$, though the function $\text{Active}$ is not assumed to be available to $H$.} then $\alpha_{n+1}$ is the smallest ordinal $>\alpha_n$ which is missing from $z^{\mathcal{X}_n}$ but is in $\cup \{ z^{\mathcal{M}},z^{\mathcal{Y}},z^{\mathcal{N}} \}$.  (This makes use of \eqref{End_Extend_active})
 \item $\sigma$ can be decoded from the parameters $\vec{\alpha}$, $z^{\mathcal{Y}}$, and $z^{\mathcal{N}}$ as follows:  for $k \in \omega$, $\sigma(k) = 1$ if $\alpha_{2k} \in z^{\mathcal{Y}}$, and $\sigma(k) = 0$ if $\alpha_{2k} \in z^{\mathcal{N}}$. 
 \end{enumerate}

Thus, since $\sigma \notin V$, it follows that there is at least one $\mathcal{X}^* \in \{ \mathcal{M},\mathcal{Y},\mathcal{N} \}$ such that $z^{\mathcal{X}^*} \notin V$.  Note that since $\text{Add}(\omega)$ is ccc, then in particular $\sigma$ automatically includes a master condition (namely $\emptyset$) for $Q^{\mathcal{X}^*}_n$, for every $n < \omega$.  It follows that $Q^{\mathcal{X}^*}_n[\sigma] \cap ORD = Q^{\mathcal{X}^*}_n \cap ORD$ for every $n \in \omega$.  Also $\langle Q^{\mathcal{X}^*}_n[\sigma] \ | \ n \in \omega \rangle$ is a $\prec$-chain of elementary submodels of $(H_\theta[\sigma], \in)$.  Notice by construction that all these models have the same intersection with $\omega_1$ also; namely $N \cap \omega_1$.  Set
\[ M:= \bigcup_{n < \omega} Q^{\mathcal{X}^*}_n[\sigma]. \]  Then $N[\sigma] \sqsubseteq M$, $M \prec (H_\theta[\sigma],\in)$, and 
\begin{equation*}
M \cap \omega_2 = \bigcup_{n < \omega} \Big( Q^{\mathcal{X}^*}_n[\sigma] \cap \omega_2 \Big) = \bigcup_{n < \omega} \Big( Q^{\mathcal{X}^*}_n \cap \omega_2 \Big) = z^{\mathcal{X}^*} \notin V.
\end{equation*}
This completes the proof of the claim.
\end{proof}

\subsection{Proof of Theorem \ref{thm_GeneralizeSakai}}

Let $\mathcal{J}$ be a normal ideal on $\omega_2$ such that $\wp(\omega_2)/\mathcal{J}$ is a proper forcing.  Fix a sufficiently large regular $\theta$ and a wellorder $\Delta$ on $H_\theta$; let 
\[ \mathfrak{A}:= (H_\theta, \in, \Delta, \{ \mathcal{J} \}). \]

Since $\mathcal{J}$-positive sets are unbounded in $\omega_2$, then to prove $\text{SCC}^{\text{cof}}_{\text{gap}}$ it suffices (by Lemma \ref{lem_WoodinLemma}) to find club-many $N \in [H_\theta]^\omega$ such that: 
\begin{equation*}
\{ \alpha < \omega_2 \ | \  \text{Sk}^{\mathfrak{A}}(N \cup \{ \alpha \}) \cap \alpha = N \cap \alpha \} \in \mathcal{J}^+ .
\end{equation*}

By Fact \ref{fact_Hulls}, for any $N \prec \mathfrak{A}$:
\begin{equation*}
\text{Sk}^{\mathfrak{A}}(N \cup \{ \alpha \}) = \{ f(\alpha) \ | \ f \in N \text{ and } f \text{ is a function } \}.
\end{equation*}
Thus it suffices to find club-many $N \in [H_\theta]^\omega$ such that
\begin{equation*}
A_N:= \{ \alpha < \kappa \ | \ \forall f \in N \ \ f(\alpha) < \alpha \implies f(\alpha) \in N \} \in \mathcal{J}^+ .
\end{equation*}

Suppose toward a contradiction that
\begin{equation*}
S:= \{ N \prec \mathfrak{A} \ | \ A_N \in \mathcal{J}   \} \text{ is stationary in } [H_\theta]^\omega .
\end{equation*}

Let $U$ be generic for $P(\omega_2)/\mathcal{J}$ and let $j: V \to_U M_U$ be the generic ultrapower embedding.  By Fact \ref{fact_includeMasterCondition}, there is some $N \in S$ (in fact stationarily many) such that $U$ includes a master condition for $N$.  Fix such an $N$ for the remainder of the proof.

Set $\kappa:= \omega_2^V$.  Since $N \in S$ then $A_N \in \mathcal{J}$, which implies that $A_N \notin U$ and thus $\kappa \notin j(A_N)$.  Also since $|N|^V < \text{crit}(j)$ then $j(N) = j[N]$; so 
\[
j(A_N) = \{ \alpha < \kappa \ | \ \forall f \in j(N)=j[N] \ \ f(\alpha) < \alpha \implies f(\alpha) \in j(N) = j[N]   \}.
\]
Since $\kappa \notin j(A_N)$ there is some $f' \in j[N]$ such that $f'(\kappa) < \kappa$ but $f'(\kappa) \notin j[N]$; say $f' = j(f)$ where $f \in N$.  Also note that $j[N] \cap \kappa = N \cap \kappa$. In summary, we have found an $f$ such that:
\begin{equation}\label{eq_To_Contradict}
f \in N,   \text{ } j(f)(\kappa) < \kappa \text{, and } j(f)(\kappa) \notin N .
\end{equation}

Set $\beta:= j(f)(\kappa)$; then $\kappa \in j\big(f^{-1}[\{ \beta \}] \big)$ and so $f^{-1}[\{ \beta \}] \in U$.  Since $U$ is a filter on $P^V(\kappa)$ then
\begin{equation*}
V[U] \models \beta \text{ is the unique ordinal such that } f^{-1}[\{ \beta \}] \in U .
\end{equation*}
Back in $V$, let $\dot{\beta}_f$ be the name which denotes the unique value for which $f$ is constant on a $\dot{U}$-measure one set, if such a thing exists.  Since $f \in N$ then we can assume $\dot{\beta}_f \in N$.  But then
\[
j(f)(\kappa) = \beta = (\dot{\beta}_f)_U \in N
\]
where the last relation is due to the fact that $U$ includes a master condition for $N$.  This contradicts \eqref{eq_To_Contradict}. \qed

\section{Proof of Corollary \ref{cor_gaps_cons_NoPrecip}}\label{sec_Cor_NS_omega_1}

In this brief section we use Theorem \ref{thm_GeneralizeSakai} to produce a model where $\text{SCC}^{\text{cof}}_{\text{gap}}$ holds, but there is no precipitous ideal on $\omega_1$.

Assume 0-pistol does not exist, and let $K$ be the core model (see Chapter 7 of \cite{ZemanBook}).  Work in $K$.  Assume $\kappa$ is a measurable cardinal and let $G$ be $(K, \text{Col}(\omega_1, < \kappa))$-generic.  By Jech-Magidor-Mitchell-Prikry~\cite{MR560220}, in $K[G]$ there is a normal ideal $\mathcal{J}$ on $\aleph_2 = \kappa$ such that $P(\kappa)/\mathcal{J}$ is forcing equivalent to a $\sigma$-closed poset.  By Theorem \ref{thm_GeneralizeSakai}, 
\begin{equation*}
K[G] \models \text{SCC}^{\text{cof}}_{\text{gap}} .
\end{equation*}

Now $K$ is absolute for set forcing; so $K$ is the core model from the point of view of $K[G]$ (\cite{ZemanBook}, Theorem 7.4.11).  If $K[G]$ had a precipitous ideal on $\omega_1^{K[G]}$, then $\omega_1^{K[G]}$ would be measurable in $K$ (\cite{ZemanBook}, Theorem 7.4.8).  But since $\text{Col}(\omega_1, < \kappa)$ preserves $\omega_1$ then
\[
\omega_1^{K[G]} = \omega_1^K
\]
so it is impossible for $\omega_1^{K[G]}$ to be measurable in $K$.  So $K[G]$ satisfies $\text{SCC}^{\text{cof}}_{\text{gap}}$ but has no precipitous ideal on $\omega_1$.

\section{Some remarks about Strong Chang's Conjecture,  special Aronszajn trees on $\omega_2$, and bounded dagger principles}\label{sec_ConcludingRemarks}

We call attention to the following two theorems:
\begin{theorem}[Todorcevic-Torres Perez~\cite{MR2965421}, Theorem 2.2]\label{thm_TodPerezThm}
If CH fails and $\text{SCC}^{\text{cof}}_{\text{gap}}$ holds, then there are no \textbf{special} Aronszajn trees on $\omega_2$.
\end{theorem}

\begin{theorem}[Torres Perez-Wu~\cite{MR3431031}, Theorem 3.1]\label{thm_Wu_Thm}
If CH fails and $\text{SCC}^{\text{cof}}$ holds, then there are no Aronszajn trees on $\omega_2$ (i.e.\ the Tree Property holds at $\omega_2$).
\end{theorem}

Theorem \ref{thm_Wu_Thm} strengthens Theorem \ref{thm_TodPerezThm} by weakening the hypothesis and strengthening the conclusion.  We observe that the hypothesis of Theorem \ref{thm_TodPerezThm} can in fact be weakened all the way to SCC, and the proof actually follows via a circuitous route from several older theorems:
\begin{theorem}\label{thm_Improved_TodorcPerez}
If CH fails and SCC holds, then there are no \textbf{special} Aronszajn trees on $\omega_2$. 
\end{theorem}

Theorem \ref{thm_Improved_TodorcPerez} follows immediately from the following three facts:
 \begin{enumerate}
  \item Todorcevic (Lemma 6 of ~\cite{MR1261218}) proved that SCC implies $\text{WRP}([\omega_2]^\omega)$.
  \item $\text{WRP}([\omega_2]^\omega)$ implies---in fact is equivalent to---the non-existence of a costationary, local club subset of $[\omega_2]^\omega$;\footnote{A set $T \subseteq [\omega_2]^\omega$ is called \emph{local club} iff $T \cap [\beta]^\omega$ contains a club for every $\beta < \omega_2$.} and
   \item If there is a special Aronszajn tree on $\omega_2$ then there is a \emph{thin} local club subset $T$ of $[\omega_2]^\omega$ (Theorem 2.3 of Friedman-Krueger~\cite{MR2276627});\footnote{$T$ is thin if for every $\beta < \omega_2$:  $|\{ a \cap \beta \ | \ a \in T \}|\le \omega_1$.} and if CH fails then this $T$ must be co-stationary in $[\omega_2]^\omega$, by a result of Baumgartner-Taylor (see Theorem 2.7 of \cite{MR2276627}).
    \end{enumerate}

Note that while Theorem \ref{thm_Wu_Thm} subsumes Theorem \ref{thm_TodPerezThm} in a strong way, it does not subsume Theorem \ref{thm_Improved_TodorcPerez} because it uses $\text{SCC}^{\text{cof}}$ instead of just SCC.  In fact the proof of Theorem \ref{thm_Wu_Thm} heavily uses the ``cofinal" requirement in the definition of $\text{SCC}^{\text{cof}}$.

Finally, we remark on a theorem of Usuba.  Let $\dagger_{\omega_2}$ abbreviate the statement:  every poset of size $\le \omega_2$ which preserves stationary subsets of $\omega_1$ is semiproper.  Let 
\[
\mathbb{Q}:= \text{Add}(\omega) * \dot{\mathbb{C}}\big([\omega_2]^\omega - V \big).
\]
We observe:
\begin{equation}\label{eq_UsubaDaggerImpliesOurs}
\dagger_{\omega_2} \implies \mathbb{Q} \text{ is semiproper.}
\end{equation}
To see this, first observe that $\mathbb{Q}$ \emph{always} has the following properties:
\begin{itemize}
 \item it preserves stationary subsets of $\omega_1$ (see \cite{MR2276627}); and
 \item it has cardinality $\text{max}\{ \omega_2, 2^\omega   \}$.
\end{itemize}
Now Usuba (Theorem 1.7 of \cite{MR3248209}) proved that $\dagger_{\omega_2}$ implies $2^\omega \le \omega_2$.\footnote{And SCC, though we won't use that.}  So if $\dagger_{\omega_2}$ holds then in particular
\[
|\mathbb{Q}| = \text{max}\{ \omega_2, 2^\omega   \} = \omega_2
\]
and since $\mathbb{Q}$ preserves stationary subsets of $\omega_1$ then $\dagger_{\omega_2}$ applies to it.  Thus $\mathbb{Q}$ is semiproper.

\end{document}